\documentclass[a4paper,11pt]{article}
\usepackage{etex}
\usepackage[a4paper,left=24mm,right=24mm,top=24mm,bottom=24mm,marginpar=24mm]{geometry}
\usepackage[utf8x]{inputenc}
\usepackage[english]{babel}
\usepackage{amsmath}
\usepackage{amssymb}
\usepackage{amsthm}
\usepackage{enumerate}
\usepackage{empheq}
\usepackage{fancybox}
\usepackage{xcolor}
\usepackage{mathrsfs}
\usepackage{slashed}
\usepackage{tikz}
\usepackage{harpoon}
\usepackage{setspace}
\usepackage{fancybox}
\usepackage{esint}
\usepackage{bm}
\usepackage[T1]{fontenc}
\usepackage{subfigure}
\usepackage[symbol]{footmisc}

\usepackage{color}
\definecolor{hanblue}{rgb}{0.27, 0.42, 0.81}
\definecolor{red}{rgb}{1.0, 0.0, 0.0}
\usepackage[colorlinks, citecolor=hanblue,linkcolor=red, urlcolor = hanblue]{hyperref}

\newcommand{\res}{\mathop{\hbox{\vrule height 7pt width .5pt depth 0pt
\vrule height .5pt width 6pt depth 0pt}}\nolimits}

\newcommand{\R}{\mathbb{R}}
\newcommand{\Z}{\mathbb{Z}}

\newcommand{\N}{\mathbb{N}}

\newcommand{\Ha}{\mathcal{H}}

\newcommand{\astfootnote}[1]{
\let\oldthefootnote=\thefootnote
\setcounter{footnote}{0}
\renewcommand{\thefootnote}{$\bigstar$}
\footnote{#1}
\let\thefootnote=\oldthefootnote
}

\newtheorem{thm}{Theorem}[section]

\newtheorem{cor}[thm]{Corollary}

\theoremstyle{definition}
\newtheorem{defi}[thm]{Definition}

\newtheorem{pro}[thm]{Problem}

\theoremstyle{remark}
\newtheorem{rmk}[thm]{Remark}

\newtheorem{ex}[thm]{Example}

\begin{document}

\title{The oriented mailing problem and its convex relaxation}
\author{Marcello Carioni, Andrea Marchese, Annalisa Massaccesi, Alessandra Pluda, Riccardo Tione \astfootnote{The second and third author received partial support from GNAMPA-INdAM. The research of
the third author has been supported by European Union’s Horizon 2020 programme through project 752018.}}

\maketitle

\begin{abstract}

In this note we introduce a new model for the mailing problem in branched transportation
that takes into account the orientation of the moving particles. This gives an effective answer to~\cite[Problem 15.9]{Bernot2009}.
Moreover we define a convex relaxation in terms of rectifiable currents with group coefficients. We provide the problem with a notion of calibration. Using similar techniques we define a convex relaxation and a corresponding notion of calibration for a variant of the Steiner tree problem in which a connectedness constraint is assigned only among a certain partition of a given set of finitely many points.
\end{abstract}



\section{Introduction}

A common feature of some natural transportation problems is the tendency to group the mass particles during the transportation which leads to the formation of branched structures. Several equivalent models for the so--called branched transportation problem have been proposed (see e.g. \cite{Bernot2009,buttazzopratelli,Maddalena2003, santambrogio, Xia2003, Xia2005}). Via an Eulerian formulation (\cite{Xia2003}) one can describe the flow of particles and via a Lagrangian one (\cite{Maddalena2003}) one can describe the trajectory of each particle. The mailing problem is a branched transportation problem with the additional constraint that, for every mass particle, an initial position and a target destination are prescribed. Only the Lagrangian formulation provides a suitable description of the problem. In the classical models, the cost functional is obtained integrating a concave function -- depending on the total amount of mass passing through each point -- on the network created by the particles' trajectories. This is not satisfactory for many concrete problems, since it is reasonable to treat differently the particles flowing in one direction on a stretch of the network and the particles flowing in the opposite direction. The necessity of a model taking into account this feature was remarked in \cite[Problem 15.9]{Bernot2009}. In this note we introduce such a model for discrete transportation networks. 

The main idea is to give a \emph{label} to each particle, containing the information on its origin and destination. In this way, we can observe a ``labeled'' flow of particles and therefore on each stretch of the network we can identify a group of particles which flows in one direction and a group which flows in the opposite one and we take this into account when we compute the cost of such transportation. This way of distinguishing the particles fits in a recent variational formulation of the so--called \emph{multi--commodity flow problem} (see \cite{Mar_Mas_Stu_Ti}), for which a convex relaxation was proposed in \cite{MMT}. Such approach was initiated in \cite{Mar_Mas1} and \cite{Mar_Mas2} in the framework of the Steiner tree problem and the Gilbert-Steiner problem, respectively, in order to tackle the difficult task of proving the optimality of a candidate minimizer. Similar approaches were recently presented in~\cite{Bonafini, Bonafini2016,Bo_Ou, Wolanski}. More precisely, our strategy allows to prove the equivalence between the original problem and the minimization of a convex functional, which is defined on a non-convex set, though. Yet an interesting by-product is the possibility to define a notion of calibration, which proved to be an efficient tool to validate the minimality both from the theoretical (\cite{MMT,   Mar_Mas1, Mar_Mas2,   Ca_Plu, coveringbis}) and from the numerical point of view (\cite{Bonafini, Bonafini2016, Bo_Ou, Mas_Ou_Ve}).

In the last section of the paper, we employ similar techniques to define a convex relaxation of a variant of the Steiner tree problem. Given a set of points $S=\{p_1,\dots,p_n\}\subset\R^d$ and a partition $S_1,\dots,S_k$ of $S$, we seek a compact set $K$ of minimal length with the property that for every $i=1,\dots,k$ the points of $S_i$ are connected in $K$. We call this the \emph{partitioned Steiner tree problem}. The minimizers can have $k$ connected components or less,
depending on the position of the given points and the partition. We show with a simple example that the existence of a calibration is not a necessary condition for minimality, but it is only sufficient.

\section{Preliminaries}

This section aims at fixing the essential notation. We refer the reader to~\cite[Section~3]{Mar_Mas_Stu_Ti} for a fully detailed introduction.
Most of the paper can be understood by modeling a multi--material network as a weighted oriented 
graph ${\mathcal G}$ in $\R^d$, with a finite set of vertices $V({\mathcal G})\subset \R^d$, a set of 
straight edges $E({\mathcal G})$, and a \emph{vector--valued} 
multiplicity function $m:E({\mathcal G})\to\R^m$. 
Being oriented, each edge $e$ has an initial vertex $e^-$ and a final vertex $e^+$. 
For every edge $e\in E({\mathcal G})$ with multiplicity $m_e$ its boundary is represented 
by the $\R^m$--valued measure 
\begin{equation*}
\partial e:=m_e\delta_{e^+}-m_e\delta_{e^-}\,,
\end{equation*}
where $\delta_x$ is the Dirac delta at the point $x$. 
Therefore the boundary of $\mathcal{G}$ is the $\R^m$--valued measure
\begin{equation*}
\partial\mathcal{G}:=\sum_{e\in E(\mathcal{G})}\partial e\,.
\end{equation*}
Observe that every $\R^m$--valued measure $\mu$ can be naturally associated to an array of $m$ 
real--valued measures $\mu_1,\dots,\mu_m$. For every real--valued measure $\nu$ we denote by
$\nu^+,\nu^-$ respectively its positive and negative part, namely the positive measures 
$\nu^+:=\frac{1}{2}(|\nu|+\nu)$, $\nu^-:=\frac{1}{2}(|\nu|-\nu)$, where $|\nu|$ is the total variation 
measure of $\nu$.

It is immediate to check that if $\mu=\partial \mathcal{G}$ for some graph $\mathcal{G}$ in $\R^d$ 
with multiplicities in $\R^m$, then the measures $\mu_i^+$ and $\mu_i^-$ have the same total 
mass, for every $i=1,\dots,m$, as it happens for the positive and negative part of the boundary of (single-material) \emph{discrete mass fluxes} (see~\cite{Bran_Wirth17}).

The main novelty of the multi--material setting is that with vector--valued multiplicities we are 
allowed to consider a cost ${\mathcal C}:\R^m\to[0,+\infty)$ which distinguishes among different 
contributions and registers possibly different interactions between $m$ different materials, 
represented by the different coordinates of $\R^m$. More precisely, we define the \emph{energy} of 
a graph ${\mathcal G}$ as the weighted sum
\begin{equation*}
{\mathbb E}({\mathcal G}):=\sum_{e\in E({\mathcal G})} {\mathcal C}(m(e)){\rm Length}(e)\,.
\end{equation*}
Under suitable assumptions on the cost functional $\mathcal{C}$, it is reasonable to consider the 
variational problem of minimizing $\mathbb{E}(\mathcal{G})$ among all graphs with given 
boundary $\partial{\mathcal{G}}$. As it is common in Calculus of Variations it is more convenient to 
look for solutions in a larger class, which enjoys compactness properties. Therefore we introduce the more general notion 
of rectifiable 1--current with coefficients in $\Z^m$. We summarize below the necessary tools. 
We refer to~\cite{Mar_Mas_Stu_Ti, MMT,Mar_Mas1} for a detailed presentation of the topic.

\subsection{Rectifiable currents with coefficients in $\Z^m$}
Consider $\mathbb{R}^m$ endowed with a norm $\Vert\cdot\Vert$
and call  $\|\cdot\|^{\ast}$ the dual norm. 
Let $k\leq d$. For the purposes of this paper it suffices to consider the cases $k=0,1,2$. 
We denote $\Lambda_k(\R^d)$ the space of $k$--vectors in $\R^d$.
\begin{defi}[$\R^m$--valued $k$--covectors]
An $\R^m$--valued $k$--covector on $\R^d$ is a bilinear map 
\begin{equation*}
\omega : \Lambda_k(\R^d)\times \R^m \rightarrow \R\,.
\end{equation*}
We denote by $\Lambda^k_{\R^m}(\R^d)$ the space of 
$\R^m$--valued $k$--covectors on $\R^d$ and we endow it
with the norm
\begin{equation*}
\vert\omega\vert_c := \sup \{\|\omega(\tau,\cdot)\|^{\ast} :
 \tau\in \Lambda_k(\R^d)  \mbox{ is simple},\ |\tau| \leq 1\}\,.
\end{equation*}
An $\R^m$--valued differential $k$--form is a map 
\begin{equation*}
\omega : \R^d \rightarrow \Lambda^k_{\R^m}(\R^d)\, .
\end{equation*}
We denote by $C_c^\infty(\R^d, \Lambda^k_{\R^m}(\R^d))$ 
the space of smooth and compactly supported $\R^m$--valued differential $k$--forms. 
On this space one can consider the \emph{comass norm}
\begin{equation*}
\|\omega\|_{c} := \sup_{x\in \R^d} \vert\omega(x)\vert_c\,.
\end{equation*}
\end{defi}

The space $C_c^\infty(\R^d, \Lambda^k_{\R^m}(\R^d))$ is naturally endowed 
with a locally convex topology, built in analogy with the topology on the space 
of test functions with respect to which distributions are dual.

\begin{defi}[$k$--dimensional currents with coefficients in $\R^m$]
A $k$--dimensional current with coefficients in $\R^m$ is a linear map
\begin{equation*}
T:C_c^\infty(\R^d, \Lambda^k_{\R^m}(\R^d)) \rightarrow \R\, ,
\end{equation*}
which is continuous with respect to the topology mentioned above.
\begin{itemize}
\item For $k>0$, the \emph{boundary} of a $k$--current $T$ is a $(k-1)$--dimensional current 
with coefficients in $\R^m$, defined through the relation
\begin{equation*}
\partial T(\varphi):=T({\rm d}\varphi)\quad\forall\, \varphi\in 
C^\infty_c(\R^d, \Lambda^k_{\R^m}(\R^d))\,.
\end{equation*}
\item The \emph{mass} of a $k$--current $T$ is the quantity 
\begin{equation*}
\mathbb{M}(T) := \sup_{\|\omega\|_{c} \leq 1} T(\omega)\,.
\end{equation*}
\end{itemize} 
We define a $k$--rectifiable current with coefficients in $\Z^m$ as 
a $k$--current with coefficients in $\R^m$ that 
admits the following representation:
\begin{equation}\label{e:def_rc}
T(\omega) = \int_{\mathcal{M}} \omega(\tau,\theta)\, \mathrm{d}\Ha^k\,,
\end{equation}
where $\mathcal{M}$ is a $k$--rectifiable set on $\R^d$, $\tau \in \Lambda_k(\R^d)$ 
is a simple unit vector orienting the approximate tangent space to $\mathcal{M}$ and 
$\theta:\mathcal{M} \rightarrow \Z^m$ is the multiplicity. 
We denote such a current by $T=[\mathcal{M},\tau,\theta]$.
\end{defi}

\begin{rmk}[Mass of rectifiable $k$--currents with coefficients in $\Z^m$]
When $T$ is a rectifiable $k$--current with coefficients in $\Z^m$ 
and is represented as in~\eqref{e:def_rc}, one can check that
\begin{equation}\label{e:mass}
\mathbb{M}(T) = \int_{\mathcal{M}} \|\theta\|\,\mathrm{d}\Ha^k\,.
\end{equation}
\end{rmk}

As mentioned above, the setting of rectifiable currents 
with coefficients in $\Z^m$ with equi--bounded masses and masses of the boundaries is closed.
We refer to~\cite[Theorem 1.10]{MMT} for further details.

\subsection{Calibrations}\label{subsec:calibrations}

Following again~\cite{Mar_Mas1} we introduce  a notion of calibration
for the mass minimization problem for $1$--rectifiable currents with coefficient in $\mathbb{Z}^m$
with a prescribed boundary.

\begin{defi}[Calibration]\label{calibration}
Given $T = [\mathcal{M}, \tau,\theta]$ a 1--rectifiable current
with coefficients in $\Z^m$,
a \emph{calibration} for $T$ is a $1$--form 
$\omega \in C_c^\infty(\R^d, \Lambda_{\R^m}^1(\R^d))$ 
that satisfies the following properties:
\begin{itemize}
\item [i)] $d\omega = 0$;
\item [ii)] $\|\omega\|_{c} \leq 1$;
\item [iii)] $\omega(\tau(x),\theta(x)) = \|\theta(x)\|$ for every $x\in\mathcal{M}$.
\end{itemize}
\end{defi}

If $\omega$ is a calibration for $T=[\mathcal{M},\tau,\theta]$, then $T$ is a solution of the mass-minimization problem with prescribed boundary $\partial T$ (see~\cite[Proposition 3.2]
{Mar_Mas1}). Moreover, $T$ minimizes the mass among all 
1--currents with coefficients in $\R^m$ with the same boundary. 

\section{Oriented mailing problem}\label{WGW}

Let $S$ be a finite collection of points $\{p_1,\ldots,p_n\}\subset \R^d$.
We prescribe the ``amount of mass'' that has to be transported 
from $p_i$ to $p_j$ (and from $p_j$ to $p_i$) by 
a matrix $G=(g_{ij})_{ij}$, 
that is a matrix in $\N^{n\times n}$ 
where each entry $g_{ij}$ represents the mass flowing 
from the the point $p_i$ to the point $p_j$ for every $i\neq j$ (and we set  $g_{ii} = 0$).

The class of admissible transportation networks can be described as a finite family $\mathcal{F}$ of trajectories (namely oriented, simple, Lipschitz paths), characterized by the following property: for every  $(i,j)\in\{1,\dots,n\}\times\{1,\dots,n\}$, the family $\mathcal{F}$ contains exactly $g_{ij}$ paths, possibly repeated, $\gamma_{ij}^1,\dots,\gamma_{ij}^{g_{ij}}:[0,1]\to\R^d$, such that $\gamma_{ij}^{\ell}(0)=p_i$ and $\gamma_{ij}^{\ell}(1)=p_j$, for every $\ell=1,\dots,g_{ij}$. If the above property holds, we shortly say that $\mathcal{F}$ is \emph{compatible} with the given set $S$ and the matrix $G$.

We associate to the family $\mathcal{F}$ the 1-rectifiable set 
$$\Gamma_{\mathcal{F}}:=\bigcup_{(i,j)\in\{1,\dots,n\}^2}\bigcup_{\ell=1}^{g_{ij}}\{\gamma_{ij}^{\ell}([0,1])\}$$
and we endow $\Gamma_{\mathcal{F}}$ with an arbitrary orientation $\sigma$ (which is defined $\Ha^1$-a.e.). Let us trace the pointwise flow of the particles in both directions along the network by defining at $\Ha^1$-a.e. $x\in\Gamma_{\mathcal{F}}$ the pair
$(\theta_{\mathcal{F}}^-(x),\theta_{\mathcal{F}}^+(x))$, where
$$\theta_{\mathcal{F}}^\pm(x):=\Ha^0\{\gamma\in\mathcal{F}:{\mbox{sign}}(\langle\gamma'(x),\sigma(x)\rangle)=\pm 1\},$$
with the small abuse of notation of denoting by $\gamma'(x)$ the vector $\gamma'(t_x)$, where $\gamma(t_x)=x$; remember that the paths in $\mathcal{F}$ are simple, hence for every $\gamma\in\mathcal{F}$ and for $\Ha^1$-a.e. $x\in\Gamma_\mathcal{F}$ there exists at most one such time $t_x$. Heuristically, $\theta_{\mathcal{F}}^+(x)$ represents the total number of masses flowing at $x$ with the same orientation as $\sigma(x)$ and $\theta_{\mathcal{F}}^-(x)$ the number of those flowing in the opposite orientation.

In order to define the cost of the family $\mathcal{F}$, we fix $\alpha\in[0,1]$ and we consider $\phi$ a symmetric, monotone norm on $\R^2$ (i.e. $\phi(x,y)=\phi(y,x)$ and $\phi(x,y)\leq \phi(z,w)$, whenever $0\leq x\leq z$ and $0\leq y\leq w$) and we define  
\begin{equation}\label{e:energy_mailing_1}
\mathbb{E}^\phi_\alpha(\mathcal{F}):=\int_{\Gamma_{\mathcal{F}}}\phi(\theta^-_{\mathcal{F}}(x)^\alpha,\theta^+_{\mathcal{F}}(x)^\alpha) d\Ha^1(x).
\end{equation}
Observe that by the symmetry of $\phi$ the energy $\mathbb{E}^\phi_\alpha$ is well-defined, namely it does not depend on the orientation $\sigma$ chosen on $\Gamma_\mathcal{F}$.

We consider the following problem.
\begin{pro}[Oriented mailing problem - first version]\label{omp1}
Let $S=\{p_1,\ldots,p_n\}\subset \R^d$ and let $G$ be a matrix in $\N^{n\times n}$. Find a family $\tilde{\mathcal{F}}$ which is compatible with $S$ and $G$ such that
\begin{equation*}
\mathbb{E}^\phi_\alpha(\tilde{\mathcal{F}})=\inf\{\mathbb{E}^\phi_\alpha(\mathcal{F}) : \mathcal{F} \mbox{ is compatible with } S \mbox{ and } G\}\,.
\end{equation*}
\end{pro}

\begin{rmk}[Versatility of the cost functional]\label{rmk:models}
The energy that we defined is sufficiently flexible to fit several models. For example the choice $\alpha\in(0,1)$ and $\phi=\|\cdot\|_{p}$, with $p=\alpha^{-1}$, recovers the notion of \emph{$\alpha$--mass} typical of the classical (non--oriented) mailing problem. Instead, $\alpha\in(0,1)$ and $\phi=\|\cdot\|_1$ is particularly interesting to model the transportation of goods on trucks or post mails: indeed, in this case the energy of a two way road is simply the sum of the energy of each single road line, hence it is convenient to group different goods when they travel in the same direction, but there is no convenience in grouping goods flowing in opposite directions.
\end{rmk}

\begin{rmk}[Relation with previous models]\label{rmk:relations}
 Even though this problem is very realistic, the models which are currently available in the literature do not really fit for its description. Trying to describe it with the standard Lagrangian formulation for the mailing problem, for instance, one should at least modify the classical notion of uniform convergence of the trajectories, in order to justify the lower-semicontinuity of an energy which must keep track of the orientation of the curves. Instead, trying to describe it with genuinely real-valued currents, besides losing the information on the trajectories, one would also face cancellations due to the two directions of movement. This originated our idea to recast the problem in the \emph{hybrid} formulation of Problem \ref{notmassmin} employing currents with coefficients in groups.
\end{rmk}

\begin{rmk}[Continuous model]
We described Problem~\ref{omp1} only when the given datum is discrete. The mailing problem, instead, is naturally defined for general (possibly diffuse) measures. In this case one could approximate the problem with rescaled discrete problems for which we can define the convex relaxation described in Section 4. Clearly one expects that solutions to the discrete problems converge to a solution of the original problem. This is in general a delicate issue. For more details on this, 
we refer the reader to~\cite{Colombo2017, Colomboa, Colombob}. We believe that the continuous version of the oriented mailing problem can be rephrased in terms of rectifiable currents with coefficients in an infinite dimensional Banach space and a corresponding convex relaxation is available in a similar framework. However, this goes beyond the purposes of the present note. 
\end{rmk}

As we already mentioned in the introduction, it is an unsolved problem in the field to propose a model which permits to discriminate different directions, allowing different costs for two way roads and one way roads which have the same total amount of traffic (see \cite[Problem 15.9]{Bernot2009} and the subsequent discussion). In the present note, we tackle this problem in the natural discrete version, by showing the existence of solutions for Problem \ref{omp1}. In order to do so, we will recast Problem \ref{omp1} in the framework of the multi-material transport problems introduced in \cite{MMT}. 

 To this aim we consider a set $S$ and a matrix $G$ as above and for  $t=1,\ldots,n$ we introduce the matrices $G^t = (g_{ij}^t)_{ij}$, where
\begin{displaymath}
g^t_{ij} := \left\{
\begin{array}{ll} 
-g_{ij} & \mbox{ for }  i=t \\
g_{ij} & \mbox{ for } j=t \\
0 & \mbox{ in all other cases }
\end{array}
\right. 
\end{displaymath}
and we define the $0$-dimensional rectifiable current with coefficients in $\Z^{n \times n}$:
\begin{equation}\label{bordo_mailing_1}
B = \sum_{t=1}^n G^t \delta_{p_t}\,.
\end{equation}

\begin{rmk}
For each $t\in\{1,\ldots,n\}$, in the $t^{\rm th}$ row of the matrix $G^t$, we are keeping track of the mass which is flowing from $p_t$ towards the other points. More precisely, in the $j^{\rm th}$ entry of the $t^{\rm th}$ row one can find the amount of mass which is supposed to flow from $p_t$ to $p_j$. Similarly, in the $i^{\rm th}$ entry of the $t^{\rm th}$ column of $G^t$ one reads the incoming mass from $p_i$ to $p_t$. Since a current with coefficients in $\Z^{n\times n}$ can be regarded as an array of $n^2$ classical integral currents with different labels (see the definition of \emph{components of a current} in \cite{MMT}), then every rectifiable 1-current $T$ with coefficients in $\Z^{n\times n}$, which satisfies $\partial T=B$, is the superposition of $n^2$ labeled mass fluxes $T_{ij}$ (in the sense of \cite{Bran_Wirth17}). Each $T_{ij}$ transports the measure $g_{ij}\delta_{p_i}$ onto the measure $g_{ij}\delta_{p_j}$.
\end{rmk}

We consider the symmetric, monotone norm $\phi$ on $\R^2$ and the real number $\alpha \in [0,1]$ introduced above and we define the \emph{cost functional} $\mathcal{C}:\Z^{n\times n}\to \R$ as
$$
\mathcal{C}(\theta):=\phi\left(\Big( \sum_{\theta_{ij}>0}\theta_{ij}\Big)^\alpha
,\Big\vert \sum_{\theta_{ij}<0} \theta_{ij}\Big\vert^\alpha\right).
$$
For a $1$--rectifiable current with coefficients in $\Z^{n\times n}$ $T = [\mathcal{M}, \tau, \theta]$, with $\partial T=B$ we define the \emph{$(\alpha,\phi)$--energy} of $T$ as
\begin{equation}\label{wgwenergy}
\mathbb{E}^\phi_\alpha(T):=\int_{\mathcal{M}}\mathcal{C}(\theta)\,d\Ha^1\,.
\end{equation}
We consider the following problem.
\begin{pro}[Oriented mailing problem - second version]\label{notmassmin}
Let $B$ be as in~\eqref{bordo_mailing_1}. Find a $1$--rectifiable current $\widetilde{T}$ with coefficients in $\Z^{n\times n}$ such that $\partial\widetilde T=B$ and 
\begin{equation*}
\mathbb{E}^\phi_\alpha(\widetilde T)=\inf\{\mathbb{E}^\phi_\alpha(T) : T\;\text{is a $1$--rectifiable current with coefficients in}\;
 \Z^{n\times n}\ \text{and}\ \partial T= B\}\,.
\end{equation*}
\end{pro}
\begin{rmk}[Multi--material transport problems]
Observe that $\mathcal{C}$ is a multi-material cost, in the sense of \cite[Definition 2.1]{MMT}, namely it is an even, increasing and subadditive function of $\theta$. Hence Problem~\ref{notmassmin} is a multi--material transport problem in the sense of \cite{MMT}.
\end{rmk}

The existence of solutions to Problem~\ref{notmassmin} is proved in \cite[Theorem 2.3]{MMT}
and it is obtained via the direct method in the Calculus of Variation: the proof is a standard application of the Closure Theorem for classical integral currents and the lower semicontinuity 
of the energy $\mathbb{E}^\phi_\alpha$ (see~\cite{Mar_Mas_Stu_Ti,Colombo2017a}).

As a consequence, the existence of a solution to Problem \ref{omp1} follows from its equivalence with Problem \ref{notmassmin}, which is stated below.

\begin{thm}[Equivalence between Problem \ref{omp1} and Problem \ref{notmassmin}]\label{equivalenceompnot}
There is a canonical way to associate to a set $S=\{p_1,\dots,p_n\}\subset\R^d$ and a matrix $G\in\N^{n\times n}$ a boundary $B$ as in~\eqref{bordo_mailing_1} such that the following holds. Given a family $\mathcal{F}$ compatible with $S$ and $G$, which is a minimizer of Problem~\ref{omp1}, one can construct a current $T_{\mathcal{F}}=[E,\tau,\theta]$, which is a minimizer of Problem~\ref{notmassmin} for the boundary $B$. Conversely, given a current $T=[E,\tau,\theta]$ which is a minimizer of Problem~\ref{notmassmin} for the boundary $B$, one can construct a family $\mathcal{F}_T$ which is a solution to Problem~\ref{omp1} associated to the set $S$ and the matrix $G$. Moreover 
the minimal values of Problem~\ref{omp1} and Problem~\ref{notmassmin} are the same.
\end{thm}
\begin{proof}
The relation between $S$, $G$ and $B$ is defined by \eqref{bordo_mailing_1}. We divide the proof in two steps.

{\emph{Step 1.}} Consider a competitor $\mathcal{F}$ for the minimization Problem \ref{omp1} and fix an orientation $\sigma$ on $\Gamma_{\mathcal{F}}$. We want to associate to $\mathcal{F}$ a current $T_\mathcal{F}$ which is a competitor for Problem~\ref{notmassmin} for the boundary $B$, such that $\mathbb{E}^\phi_\alpha(T_\mathcal{F})\leq\mathbb{E}^\phi_\alpha(\mathcal{F})$. To this regard, for every $\gamma\in\mathcal{F}$ with $\gamma(0)=p_i$ and $\gamma(1)=p_j$, denote by $[\gamma]$ the 1-current with coefficients in $\Z^{n\times n}$ given by $\left[\gamma[0,1],\frac{\gamma'}{|\gamma'|}, E_{ij}\right]$, where $E_{ij}:=(\delta_{ij})_{\iota \kappa}$. Denote $T_\mathcal{F}=[E,\tau,\theta]$ the current $T_\mathcal{F}:=\sum_{\gamma\in\mathcal{F}}[\gamma]$. Since for every $\gamma$ we have $\partial[\gamma]=-\delta_{p_i}E_{ij}+\delta_{p_j}E_{ij}$, it is immediate to check that $\partial T_\mathcal{F}=B$. Moreover, $T_\mathcal{F}$ is supported on (a subset of) $\Gamma_\mathcal{F}$, hence in view of the monotonicity of $\phi$, in order to guarantee that $\mathbb{E}^\phi_\alpha(T_\mathcal{F})\leq\mathbb{E}^\phi_\alpha(\mathcal{F})$ it is sufficient to check that for $\Ha^1$-a.e. $x\in\Gamma_\mathcal{F}$ it holds: 
$$\theta_\mathcal{F}^-(x)\geq \left|\sum_{\theta_{ij}(x)<0}\theta_{ij}(x)\right| \mbox{ and } \theta_\mathcal{F}^+(x)\geq \sum_{\theta_{ij}(x)>0}\theta_{ij}(x)$$
if $\langle\tau(x),\sigma(x)\rangle>0$ (recall that $\sigma$ is a fixed orientation of $\Gamma_\mathcal{F}$) and otherwise
$$\theta_\mathcal{F}^-(x)\geq \sum_{\theta_{ij}(x)>0}\theta_{ij}(x) \mbox{ and } \theta_\mathcal{F}^+(x)\geq \left|\sum_{\theta_{ij}(x)<0}\theta_{ij}(x)\right|.$$
To check the latter property observe that for every $(i,j)$ and for $\Ha^1$-a.e. $x\in\Gamma_{\mathcal{F}}$, we have 
\begin{equation}\label{spezzamento}
\theta_{ij}(x)= {\mbox{sign}}(\langle \sigma(x),\tau(x)\rangle) \left(\Ha^0(\mathcal{F}_{ij}^+(x))-\Ha^0(\mathcal{F}_{ij}^-(x))\right),
\end{equation}
where
$$\mathcal{F}_{ij}^\pm(x):=\{\gamma\in\mathcal{F}:\gamma(0)=p_i, \gamma(1)=p_j, {\mbox{ and }} x\in\gamma([0,1]) {\mbox { with}}{\mbox{ sign}}(\langle\gamma'(x),\sigma(x)\rangle)=\pm 1\}.$$
On the other hand, for $\Ha^1$-a.e. $x\in\Gamma_{\mathcal{F}}$, it holds
$$\theta_\mathcal{F}^\pm(x)=\sum_{i,j}\Ha^0(\mathcal{F}_{ij}^\pm(x)).$$
{\emph{Step 2.}} Consider a competitor $T$ for the minimization Problem \ref{notmassmin}. We want to associate to $T$ a family $\mathcal{F}_T$ which is compatible with $S$ and $G$, such that $\mathbb{E}^\phi_\alpha(\mathcal{F}_T)\leq\mathbb{E}^\phi_\alpha(T)$. By \cite[Theorem 3.2]{Mar_Mas2} we can write $T$ as a sum of a cycle plus finitely many 1-currents associated to simple, open, Lipschitz curves with some multiplicities $E_{ij}$. Moreover for every fixed $(i,j)$ the number of open curves with multiplicity $E_{ij}$ is precisely $g_{ij}$, they all go from $p_i$ to $p_j$ and they have $\Ha^1$-a.e. the same orientation when they intersect.

Additionally, by \cite[Formula (3.3)]{Mar_Mas2} it is easy to observe that the 1-current $T'$ obtained by summing over $(i,j)$ all those open curves (and hence neglecting the remaining cycles) satisfy $\mathbb{E}^\phi_\alpha(T')\leq\mathbb{E}^\phi_\alpha(T)$. If we consider the family $\mathcal{F_T}$ obtained as a union of the corresponding open curves, it follows that $\mathcal{F}_T$ is compatible with $S$ and $G$. In view of \eqref{spezzamento}, the fact that all curves going from $p_i$ to $p_j$ have $\Ha^1$-a.e. the same orientation when they intersect guarantees additionally that for $\Ha^1$-a.e. $x$ and for every $(i,j)$ at most one between $\mathcal{F}_{ij}^+(x)$ and $\mathcal{F}_{ij}^-(x)$ is non empty and therefore $\mathbb{E}^\phi_\alpha(T')=\mathbb{E}^\phi_\alpha(\mathcal{F}_T)$, which concludes the proof.
\end{proof}
\begin{cor}
The Problem \ref{omp1} admits a solution. 
\end{cor}

\section{Convex relaxation}\label{s:conv}
In this section we show how to rephrase the Problem \ref{notmassmin} as a mass--minimization problem, that is, the minimization of a convex functional. This is described in~\cite[Theorem 2.4]{MMT}.
In the following we only explain how to define the mass--minimization problem associated to Problem~\ref{notmassmin}, and we refer the reader to~\cite{MMT} for the proof of the equivalence.

Let $N:=\sum_{(i,j)}g_{ij}$. We will define as boundary datum a 0--current $\mathcal{B}$ with coefficients in $\Z^N$. Firstly we choose an ordering for the pairs $(i,j)\in\{1,\dots,n\}\times\{1,\dots,n\}$. In order to keep the notation short, we denote by $I_1,\dots, I_{n^2}$ such pairs. For every $I=(i,j)$ we denote $\mu_I$ the signed atomic measure $\mu_I:=\delta_{p_j}-\delta_{p_i}$. Then we associate to every $I=I_1,\dots, I_{n^2}$ an element $\theta_I$ of $\Z^N$ as follows. For $I=(i,j)$, we denote $g_I:=g_{ij}$. We take $\theta_{I_1}$ the sum of the first $g_{I_1}$ elements of the basis $(e_1,\dots,e_N)$, then we take $\theta_{I_2}$ the sum of the next $g_{I_2}$ elements of the basis, and so on\dots Lastly, we define
\begin{equation}\label{bordo_mailing_2}
\mathcal{B}:=\sum_{\ell=1}^{n^2}\theta_{I_\ell}\mu_{I_\ell}.
\end{equation}
 Now we define the following monotone norm on $\R^N$, where we denote $p:=\frac{1}{\alpha}$:
$$\|(t_1,\dots,t_N)\|_{\phi,\alpha}:=\phi\left(\Big\|\sum_{t_\iota>0}t_\iota e_\iota\Big\|_{\ell^p},\Big\|\sum_{t_\iota<0}t_\iota e_\iota\Big\|_{\ell^p}\right).$$
We consider the following problem, where the mass of a current with coefficients in $\Z^{N}$ is computed with respect to the norm $\|\cdot\|_{\phi,\alpha}$.
\begin{pro}[Convex relaxation of the oriented mailing problem]\label{massmin}
Let $\mathcal{B}$ be as in~\eqref{bordo_mailing_2}. Find a $1$--rectifiable current $\widetilde{R}$ with coefficients in $\Z^{N}$ such that $\partial\widetilde R=\mathcal{B}$ and 
\begin{equation*}
\mathbb{M}(\widetilde R)=\inf\{\mathbb{M}(R) : R\;\text{is a $1$--rectifiable currents with coefficients in}\;
 \Z^{N}\ \text{and}\ \partial R= \mathcal B\}\, .
\end{equation*}
\end{pro}

\begin{rmk}[A key property of the norm]\label{rem:normcost}
Observe that if $|t_\iota|\in\{0,1\}$ for every $\iota$, then 
$$\|(t_1,\dots,t_N)\|_{\phi,\alpha}=\phi(\sharp\{\iota:t_\iota=1\}^\alpha,\sharp\{\iota:t_\iota=-1\}^\alpha).$$
Notice that this implies the validity of~\cite[(3.1)]{MMT}, which is the crucial identity to ensure the validity of the forthcoming Theorem~\ref{equi}, stating the equivalence between Problem~\ref{notmassmin} and Problem~\ref{massmin}. Observe that it is not necessary to verify that $\mathcal{C}$ satisfies property $(iii)'$ of \cite[Definition 2.1]{MMT}. Indeed such property was used only in \cite[Theorem 3.2]{MMT} to prove the existence of a monotone norm satisfying~\cite[(3.1)]{MMT}, for which we already gave an explicit formula.
\end{rmk}

\begin{thm}[Equivalence between Problem~\ref{notmassmin} and Problem~\ref{massmin}]\label{equi}
There is a canonical way to associate to a boundary $B$ as in~\eqref{bordo_mailing_1} a boundary $\mathcal{B}$ as in~\eqref{bordo_mailing_2} such that the following holds. For every $T=[E,\tau,\theta]$, minimizer of Problem~\ref{notmassmin} for the boundary $B$, there is a canonical current $R_T=[E',\tau',\theta']$ which is a minimizer of Problem~\ref{massmin} for the boundary $\mathcal{B}$. Conversely, for every $R=[E',\tau',\theta']$ minimizer of Problem~\ref{massmin} for the boundary $\mathcal{B}$, there is a canonical current $T_R=[E,\tau,\theta]$ which is a minimizer of Problem~\ref{massmin} for the boundary $B$. Moreover it holds $\Ha^1(E\triangle E')=0$ and $\mathcal{C}(\theta(x))=\|\theta'(x)\|_{\phi,\alpha}$ for $\Ha^1$-a.e. $x\in E$. In particular the minimal values of Problem~\ref{notmassmin} and Problem~\ref{massmin} are the same.
\end{thm}

\begin{rmk}[Calibrations for the oriented mailing problem] Reformulating Problem \ref{notmassmin} as a mass minimization problem allows to introduce a related notion of calibration as described in Section \ref{subsec:calibrations}. This gives a useful tool to certify the minimality of a given candidate minimizer. It also provides a numerical method to estimate the energy gap between any competitor and a minimizer, as one can see in \cite[Definition 2.9 and Proposition 2.11]{Mas_Ou_Ve}. We refer to \cite{Mas_Ou_Ve} also for the details and the description of the numerical implementation.
\end{rmk}

\section{Partitioned Steiner tree problem}\label{Steiner}

We consider a finite family $S:= \{p_1,\ldots,p_n\}$ of points in $\R^d$ 
and a partition of $S$ denoted by $S_1,\ldots,S_k$. Without loss of generality we assume that in the partition there are no singletons.
We define the following problem, which we call the \emph{partitioned Steiner tree problem}
\begin{pro}\label{Steinertype}
Find a compact set $K$ of minimal $\Ha^1$--measure that contains $S$
and such that  for $i\in\{1,\ldots,k\}$ the points of $S_i$ 
are in the same connected component of $K$.  
\end{pro}

If $K$ is a competitor for Problem~\ref{Steinertype},
then the number of its connected components is bounded from above by $k$. 
Then it is not difficult to prove existence of minimizers by
a direct method in the Calculus of Variations:
compactness comes from Blaschke Selection Theorem and
lower semicontinuity of the $1$--dimensional Hausdorff measure 
for sequences of compact sets with an equi--bounded number of connected components comes from Go\l \k{a}b semicontinuity theorem (see~\cite{falconer}).

\begin{rmk} 
Notice that the solutions of Problem~\ref{Steinertype} are union of $h$ disjoint minimal Steiner trees ($h\leq k$) connecting some sets $\mathscr{S}_j$, $j=1\dots,h$, where $\mathscr{S}_1,\dots, \mathscr{S}_h$ is a partition of $S$ with the property that each $S_i$ ($i=1,\dots,k$) is contained in some $\mathscr{S}_j$. Nevertheless the partition $\mathscr{S}_1,\dots,\mathscr{S}_h$ 
(and therefore the number of connected components of a minimizer $K$) is not known a priori; 
it depends on the relative positions of the points of $S$.
\end{rmk}

We rephrase now Problem~\ref{Steinertype} as a mass minimization problem among 
a family of $1$--rectifiable currents with coefficients in $\Z^{n-k}$. We set $n_0:=0$ and for $i\in\{1,\ldots,k\}$ we denote by $n_i$ the cardinality of $S_i$. Up to reordering, we may assume that $S_1=\{p_1,\dots,p_{n_1}\}$, $S_2=\{p_{n_1+1},\dots,p_{n_1+n_2}\}$, and so on\ldots 
For $i\in\{1,\ldots,k\}$, we denote $I_i:=[a_i,b_i]\cap\N$, where
$$a_i:=n_0+\dots+n_{i-1}+1,\quad b_i:=n_0+\dots+n_i-1.$$
We denote $(e_\iota)_{\iota=1,\ldots,n-k}$ the canonical basis of $\R^{n-k}$. 
For $\ell=1,\dots,n$ we define  
\begin{equation}\label{gi}
g_\ell := \left\{
\begin{array}{ll} 
 e_{\ell-i+1} & \mbox{if } \ell\in I_i, {\mbox{ for some }} i\in\{1,\dots,k\}\\
-\sum_{\iota= a_i}^{b_i} e_\iota & \mbox{if } \ell = b_i+1,{\mbox{ for some }} i\in\{1,\dots,k\}\,.
\end{array}
\right. 
\end{equation}
We set 
\begin{equation*}
B = \sum_{\ell=1}^n g_\ell \delta_{p_\ell}\,.
\end{equation*}

We want to solve the following minimization problem, where the mass of a current with coefficients in $\Z^{n-k}$ is computed with respect to the $\ell^\infty$--norm on $\R^{n-k}$.
 
\begin{pro}\label{massminpro}
Find a $1$--rectifiable current $\bar{T}$ with coefficients in $\Z^{n-k}$ such that $\partial\bar T=B$ and 
\begin{equation*}
\mathbb{M}(\bar T)=\inf\{\mathbb{M}(T) : T\;\text{is a $1$--rectifiable current with coefficients in}\;
 \Z^{n-k}\,\text{and}\, \partial T= B\}\, .
\end{equation*}
\end{pro}

The existence of minimizers for Problem~\ref{massminpro}
is granted by a direct method in the Calculus of Variations:
the lower semicontinuity of the mass is an immediate consequence of its definition and the compactness in the space of 1--rectifiable currents with coefficients in $\Z^{n-k}$ is a simple consequence of the Closure Theorem for classical integral currents (see~\cite{Mar_Mas1}  and \cite[Theorem 1.10]{MMT} for the details).

\medskip

It is then possible to prove the following equivalence result:
\begin{thm}\label{equi_St}
Given a set $K$ which is a minimizer for Problem~\ref{Steinertype}, 
there exists a $1$--rectifiable current $T_K=[K,\tau_K,\theta_K]$ with coefficients in $\Z^{n-k}$ with $\|\theta_K\|=1$ $\mathcal{H}^1$--a.e. on $K$ and $T$ is minimizer for Problem~\ref{massminpro}. On the other hand, given $T = [\mathcal{M},\tau,\theta]$ a mass minimizing current for Problem~\ref{massminpro}, its support is a minimizer for Problem~\ref{Steinertype}.
In particular $\mathbb{M}(T)=\mathcal{H}^1(K)$.
\end{thm}

This result can be obtained with minor changes from \cite[Theorem 2.4]{Mar_Mas1}. For the reader's convenience, we give a sketch of the proof.\\

\textit{Sketch of the Proof}

Given $K$ a minimizer for Problem~\ref{Steinertype}, we construct $T_K$ exploiting one crucial property of the solutions to the Steiner tree problem, namely the absence of loops in their support. We begin with the group of points $S_1=\{p_1,\dots,p_{n_1}\}$. For every point $p_j$, $j=1,\dots,n_1-1$, we consider the (unique) path $K_j$ in $K$, oriented by $\tau_j$, connecting $p_{n_1}$ to $p_j$. We denote $T_j:=[K_j,\tau_j, g_j]$ and we let $T_K^1:=\sum_{j=1}^{{n_1}-1} T_j$. One can see immediately that $\partial T_K^1=\partial T\res S_1$. 

Analogously we construct $T_K^2,\dots,T_K^k$ and we define $T^K:=\sum_{i=1}^k T_K^i$. By construction we have that the support of $T_K$ is contained in $K$, and by the choice of the norm on $\R^{n-k}$ we have $\|\theta_K\|=1$ $\mathcal{H}^1$--a.e. on $K$ and therefore, by \eqref{e:mass}, we have 
\begin{equation}\label{e:steiner1}
\mathbb{M}(T_K)= \mathcal{H}^1(K).
\end{equation}

Now fix any competitor $Z$ for Problem \ref{massminpro}. By the structure of the boundary $B$, the support of $Z$ contains a connected path from each point of $S_1\setminus\{p_{n_1}\}$ to $p_{n_1}$, and the same is true for the groups $S_2,\dots,S_k$. Hence the support of $Z$ is a competitor for Problem \ref{Steinertype}. Considering that the norm of each non--zero element of $\Z^{n-k}$ is at least one, we deduce that
$$\mathbb{M}(Z)\geq\mathcal{H}^1(K),$$
hence, by \eqref{e:steiner1}, $T_K$ is a solution to Problem \ref{massminpro}.
\medskip

The other implication is more technical. Via an operation which heuristically consists in removing all cycles from a classical integral current, we can show that for every competitor $Z=[\mathcal{M}_Z,\tau_Z,\theta_Z]$ (with $\theta_Z\neq 0$ on $\mathcal{M}_Z$) for Problem~\ref{massminpro}, one can find another competitor $Z'=[\mathcal{M}_{Z'},\tau_{Z'},\theta_{Z'}]$ with $\mathcal{M}_{Z'}\subset \mathcal{M}_{Z}$ and $\|\theta_{Z'}\|=1$ $\mathcal{H}^1$--a.e. on $\mathcal{M}_{Z'}$, and in particular $\mathcal{H}^1(\mathcal{M}_{Z'})=\mathbb{M}(Z')\leq\mathbb{M}(Z)$, with strict inequality unless $\mathcal{H}^1(\mathcal{M}_Z\setminus\mathcal{M}_{Z'})=0$ and $\|\theta_{Z}\|=1$ a.e. on $\mathcal{M}_{Z}$. This implies in particular that $\mathbb{M}(T)=\mathcal{H}^1(\mathcal{M})$. The procedure is very close to the one presented in Step 2 of Lemma \ref{equivalenceompnot}. For a proof in a similar context, see also \cite[Lemma 2.3]{Mar_Mas1}.

Assume by contradiction that $\mathcal{M}$ is not a minimizer for Problem~\ref{Steinertype} and hence there exists a competitor $\mathcal{N}$ for Problem~\ref{Steinertype} such that $\mathcal{H}^1(\mathcal{N})<\mathcal{H}^1(\mathcal{M})$. With the procedure described in the first part of the proof we can construct a competitor $T_\mathcal{N}=[\mathcal{N},\tau_\mathcal{N},\theta_\mathcal{N}]$ for Problem~\ref{massminpro} with $\|\theta_\mathcal{N}\|=1$ $\mathcal{H}^1$--a.e. on $\mathcal{N}$. This would lead to the contradiction
$$
\mathcal{H}^1(\mathcal{N})=\mathbb{M}(T_{\mathcal{N}})\geq\mathbb{M}(T)=\mathcal{H}^1(\mathcal{M}). 
$$
\qed

As an example, we propose here a simple partitioned Steiner tree problem. This also gives us the opportunity to remark that a calibration, in the sense of Definition \ref{calibration}, does not always exist.

\begin{ex}\label{ex:withfigures}
Let $S=\{p_1,\ldots,p_4\}$ with $p_1=(1,1),p_2=(1,-1),p_3=(-1,-1)$ and $p_4=(-1,1)$
and $S_1=\{p_1,p_3\}$, $S_2=\{p_2,p_4\}$. We want to minimize the mass 
among all $1$--rectifiable currents $T$ with coefficients in $\mathbb{Z}^2$
such that $\partial T=B:=\sum_i g_i\delta_{p_i}$ where $g_i$ 
are defined according to~\eqref{gi}.
\emph{A priori} the competitors of this simple problem have two connected components,
but it is more convenient if the two components ``interact''.
Hence the supports of the minimizers $T_1$ and $T_2$ are the minimal Steiner networks
connecting the four points depicted in Figure~\ref{esempio4punti} (notice that the orientations represented in Figure~\ref{esempio4punti} match the construction explained in the proof of Theorem \ref{equi_St}).
We have $\mathbb{M}(T_i)=2+2\sqrt{3}$.

\medskip


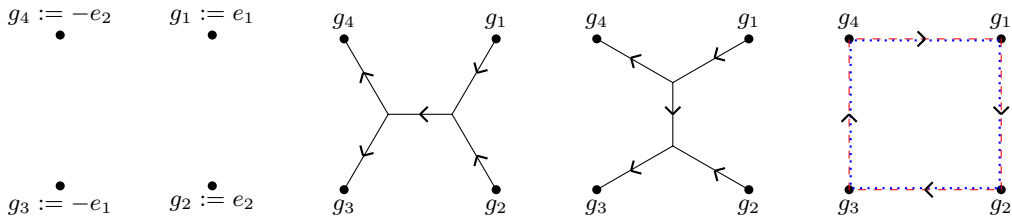
\begin{figure}[h]
\begin{center}
\begin{tikzpicture}
\path[font=\footnotesize]
(1,1) node[above]{$g_1:=e_1$}
(1,-1) node[below]{$g_2:=e_2$}
(-1,-1) node[below]{$g_3:=-e_1$}
(-1,1) node[above]{$g_4:=-e_2$};
\fill[black](1,1) circle (1.7pt);    
\fill[black](1,-1) circle (1.7pt);    
\fill[black](-1,1) circle (1.7pt);    
\fill[black](-1,-1) circle (1.7pt);    
\end{tikzpicture}\qquad
\begin{tikzpicture}
\path[font=\footnotesize]
(1,1) node[above]{$g_1$}
(1,-1) node[below]{$g_2$}
(-1,-1) node[below]{$g_3$}
(-1,1) node[above]{$g_4$};
\fill[black](1,1) circle (1.7pt);    
\fill[black](1,-1) circle (1.7pt);    
\fill[black](-1,1) circle (1.7pt);    
\fill[black](-1,-1) circle (1.7pt);    
\draw
(1,-1)--(0.42,0)
(1,1)--(0.42,0)
(0.42,0)--(-0.42,0)
(-0.42,0)--(-1,-1)
(-0.42,0)--(-1,1);
\draw[thick, shift={(0.75,0.55)}, scale=1, rotate=-210]
(0,0)to[out= -45,in=135, looseness=1] (0.1,-0.1)
(0,0)to[out= -135,in=45, looseness=1] (-0.1,-0.1);
\draw[thick, shift={(-0.75,-0.55)}, scale=1, rotate=-210]
(0,0)to[out= -45,in=135, looseness=1] (0.1,-0.1)
(0,0)to[out= -135,in=45, looseness=1] (-0.1,-0.1);
\draw[thick, shift={(0.75,-0.55)}, scale=1, rotate=30]
(0,0)to[out= -45,in=135, looseness=1] (0.1,-0.1)
(0,0)to[out= -135,in=45, looseness=1] (-0.1,-0.1);
\draw[thick, shift={(-0.75,0.55)}, scale=1, rotate=30]
(0,0)to[out= -45,in=135, looseness=1] (0.1,-0.1)
(0,0)to[out= -135,in=45, looseness=1] (-0.1,-0.1);
\draw[thick, shift={(0,0)}, scale=1, rotate=90]
(0,0)to[out= -45,in=135, looseness=1] (0.1,-0.1)
(0,0)to[out= -135,in=45, looseness=1] (-0.1,-0.1);
\end{tikzpicture}\qquad
\begin{tikzpicture}
\path[font=\footnotesize]
(1,1) node[above]{$g_1$}
(1,-1) node[below]{$g_2$}
(-1,-1) node[below]{$g_3$}
(-1,1) node[above]{$g_4$};
\fill[black](1,1) circle (1.7pt);    
\fill[black](1,-1) circle (1.7pt);    
\fill[black](-1,1) circle (1.7pt);    
\fill[black](-1,-1) circle (1.7pt);    
\draw[rotate=90]
(1,-1)--(0.42,0)
(1,1)--(0.42,0)
(0.42,0)--(-0.42,0)
(-0.42,0)--(-1,-1)
(-0.42,0)--(-1,1);
\draw[thick, shift={(0.55,0.75)}, scale=1, rotate=-240]
(0,0)to[out= -45,in=135, looseness=1] (0.1,-0.1)
(0,0)to[out= -135,in=45, looseness=1] (-0.1,-0.1);
\draw[thick, shift={(-0.55,-0.75)}, scale=1, rotate=-240]
(0,0)to[out= -45,in=135, looseness=1] (0.1,-0.1)
(0,0)to[out= -135,in=45, looseness=1] (-0.1,-0.1);
\draw[thick, shift={(0.55,-0.75)}, scale=1, rotate=60]
(0,0)to[out= -45,in=135, looseness=1] (0.1,-0.1)
(0,0)to[out= -135,in=45, looseness=1] (-0.1,-0.1);
\draw[thick, shift={(-0.55,0.75)}, scale=1, rotate=60]
(0,0)to[out= -45,in=135, looseness=1] (0.1,-0.1)
(0,0)to[out= -135,in=45, looseness=1] (-0.1,-0.1);
\draw[thick, shift={(0,0)}, scale=1, rotate=180]
(0,0)to[out= -45,in=135, looseness=1] (0.1,-0.1)
(0,0)to[out= -135,in=45, looseness=1] (-0.1,-0.1);
\end{tikzpicture}\qquad
\begin{tikzpicture}
\path[font=\footnotesize]
(1,1) node[above]{$g_1$}
(1,-1) node[below]{$g_2$}
(-1,-1) node[below]{$g_3$}
(-1,1) node[above]{$g_4$};
\fill[black](1,1) circle (1.7pt);    
\fill[black](1,-1) circle (1.7pt);    
\fill[black](-1,1) circle (1.7pt);    
\fill[black](-1,-1) circle (1.7pt);    
\draw[red,dashed]
(1,-1)--(1,1)--(-1,1)--(-1,-1)--(1,-1);
\draw[blue, thick, dotted]
(0.98,-0.98)--(0.98,0.98)--(-0.98,0.98)--(-0.98,-0.98)--(0.98,-0.98);
\draw[thick, shift={(0,-1)}, scale=1, rotate=90]
(0,0)to[out= -45,in=135, looseness=1] (0.1,-0.1)
(0,0)to[out= -135,in=45, looseness=1] (-0.1,-0.1);
\draw[thick, shift={(0,1)}, scale=1, rotate=-90]
(0,0)to[out= -45,in=135, looseness=1] (0.1,-0.1)
(0,0)to[out= -135,in=45, looseness=1] (-0.1,-0.1);
\draw[thick, shift={(-1,0)}, scale=1, rotate=0]
(0,0)to[out= -45,in=135, looseness=1] (0.1,-0.1)
(0,0)to[out= -135,in=45, looseness=1] (-0.1,-0.1);
\draw[thick, shift={(1,0)}, scale=1, rotate=180]
(0,0)to[out= -45,in=135, looseness=1] (0.1,-0.1)
(0,0)to[out= -135,in=45, looseness=1] (-0.1,-0.1);
\end{tikzpicture}
\end{center}
\caption{On the left the boundary datum $B$. In the middle the two minimizers $T_1$ and $T_2$. On the right the current $Z$.}\label{esempio4punti}
\end{figure}

We will show that there exists a 1--current $Z$ with coefficients in $\R^2$ such that $\partial Z=B$ and $\mathbb{M}(Z)<\mathbb{M}(T_i)$, $i\in\{1,2\}$. Denote by $\mathcal{M}_j$ the oriented segment joining $p_j$ to $p_{j+1}$
with $j\in\{1,2,3,4\}$ (cyclically identified) and its unit tangent vector $\tau_j$.
We define a current $Z$ with coefficients in $\mathbb{R}^2$
with support the four edges of the square
as $\sum_{j=1}^4 Z_j$  with
\begin{equation*}
\begin{array}{rlrl}
Z_1 & =[\mathcal{M}_1,\tau_1,\frac{1}{2}(e_2-e_1)],\qquad &
Z_2 & =[\mathcal{M}_2,\tau_2,-\frac{1}{2}(-e_1-e_2)],\\
Z_3 & =[\mathcal{M}_3,\tau_3,\frac{1}{2}(e_1-e_2)],\qquad &
Z_4 & =[\mathcal{M}_4,\tau_4,\frac{1}{2}(e_1+e_2)]
\end{array}
\end{equation*}
(see Figure~\ref{esempio4punti}).

Then $4=\mathbb{M}(Z)<\mathbb{M}(T_i)$.
This contradicts the existence of a calibration for $T_i$, because, as observed in Section \ref{subsec:calibrations}, if a calibration for $T_i$ existed, then $T_i$ would minimize the mass among rectifiable currents with coefficients in $\R^{2}$ having the same boundary.

\end{ex}

\bibliographystyle{abbrv}
\bibliography{CMMPT}

  \bigskip
  \footnotesize

\bigskip

  M.C:\\
  \textsc{Department of Applied Mathematics and Theoretical Physics,\\ 
  University of Cambridge,\\ 
  Wilberforce Road, UK-CB3 OWA Cambridge} \\
  \href{mc2250@maths.cam.ac.uk}{mc2250@maths.cam.ac.uk}
  
\medskip

  A.Mar:\\ 
  \textsc{Dipartimento di Matematica,\\
  Universit\`a di Trento,\\ 
  Via Sommarive 14, IT-38123 Povo (TN)}\\
   \href{andrea.marchese@unitn.it}{andrea.marchese@unitn.it}

\medskip

  A.Mas:\\ 
  \textsc{Dipartimento di Tecnica e Gestione dei Sistemi Industriali (DTG),\\
   Universit\`a di Padova,\\ Stradella S. Nicola 3, IT-36100 Vicenza}\\
  \href{annalisa.massaccesi@unipd.it}{annalisa.massaccesi@unipd.it}

\medskip

  A.P:\\
   \textsc{Dipartimento di Matematica,\\ Universit\`a di Pisa,\\ Largo Pontecorvo 5, IT-56127 Pisa}\\
 \href{alessandra.pluda@unipi.it}{alessandra.pluda@unipi.it}
 
\medskip

  R.T:\\
   \textsc{Institut f\"ur Mathematik,\\
    Universit\"at Z\"urich,\\ Winterthurerstrasse 190, CH-8057 Z\"urich}\\
\href{riccardo.tione@math.uzh.ch}{riccardo.tione@math.uzh.ch}

\end{document}